\documentclass[11pt]{article}
\usepackage{}
\usepackage{mathrsfs}
\usepackage{amsfonts}
\usepackage{mathrsfs}
\usepackage{amssymb}
\usepackage{color}
\usepackage{latexsym,bm}
\usepackage{amssymb}
\usepackage{amsmath}
\usepackage{amsthm}
\usepackage[dvips]{graphics}
\usepackage[htbp]{graphicx}
\usepackage[all]{xy}
\usepackage{epsfig,latexsym,amssymb, amsmath,amscd,amsthm,amsxtra}
\usepackage{graphicx}
\newtheorem{thm}{Theorem}[section]

\newtheorem{rmk}{Remark}[section]
\newtheorem{lem}{Lemma}[section]
\newtheorem{prop}{Proposition}[section]

\newtheorem{e.g.}[thm]{Example}

\begin{document}
\title{Global Solution for Gas-Liquid Flow of 1-D van der Waals Equation of State with Large Initial Data}
\author{ Qiaolin He$^1$,
 Ming Mei$^{2,3}$, Xiaoding Shi$^{4}$\thanks{\scriptsize{Corresponding author, shixd@mail.buct.edu.cn}}, Xiaoping Wang$^5$ \\
   \scriptsize{$^1$School of Mathematics, Sichuan University, Chengdu, 610064, China}\\
 \scriptsize{$^{2}$Department of Mathematics, Champlain College St.-Lambert, St.-Lambert, Quebec, J4P 3P2, Canada}\\
 \scriptsize{$^{3}$Department of Mathematics and Statistics, McGill University, Montreal, Quebec,H3A 2K6, Canada }\\
 \scriptsize{$^{4}$Department of Mathematics, School of Science, Beijing University of Chemical Technology, Beijing, 100029, China}\\
 \scriptsize{$^{5}$Department of Mathematics, Hong Kong University of Science and Technology, Hong Kong, China}}
\date{} \maketitle
 \noindent\textbf{Abstract}. This paper is concerned with  a diffuse interface model  for  the gas-liquid phase transition. The model consists the compressible Navier-Stokes equations with van der Waals equation of state and a modified Allen-Cahn equation. The global existence and uniqueness of strong solution with  the periodic boundary condition (or the mixed boundary condition)  in one dimensional space is proved  for large initial data. Furthermore, the  phase variable and the density of the gas-liquid mixture are proved to stay in the physical reasonable interval. The proofs are based on the elementary energy method and the maximum principle, but with new development,  where some techniques are introduced to establish the uniform bounds of the density and to treat the non-convexity of the pressure function.

\

 \noindent{\bf Keywords:}  global solution, Navier-Stokes equations, Allen-Cahn equation, gas-liquid flow, van der Waals equation of state
\

\noindent{\bf MSC:} 35M10,35Q30

\section {\normalsize Introduction and Main Result}
\setcounter{equation}{0}
In the last few decades, there have been many progresses on modelling and analysis of the multiphase and phase transition problems, in particular on the phase field models of the phenomena, see \cite{AC1979}-\cite{CH1958}, \cite{HMR2012} \cite{LT-1998} \cite{V1894}  \cite{QWS} \cite{WQS} and the references therein. In this paper, we investigate  Navier-Stokes-Allen-Cahn system proposed by  Blesgen \cite{B1999} which   describes the compressible two-phase flow with diffusive interface. The system consists  of  the compressible Navier-Stokes equations and a modified Allen-Cahn equation, and it
 is especially useful for analyzing the phase transition properties of gas-liquid flow. It allows phases to shrink or grow due to changes of density in the fluid and incorporates their transport with the current. The Navier-Stokes-Allen-Cahn system is commonly expressed as follows (see \cite{B1999}, \cite{FPRS2008}, \cite{DLL2013}, \cite{CG2017}, \cite{STY} etc.)
\begin{equation}\label{3dNSAC}
\left\{\begin{array}{llll}
\displaystyle \partial_t\rho+\textrm{div}(\rho \mathbf{u})=0,  \\
\displaystyle \displaystyle \partial_t(\rho \mathbf{u})+\textrm{div}(\rho \mathbf{u}\otimes \mathbf{u})+\nabla p-(\nu\Delta\mathbf{u}+\eta\nabla\textrm{div}\mathbf{u})=-\epsilon\mathrm{div}\big(\nabla\chi\otimes\nabla\chi-\frac{|\nabla\chi|^2}{2}\mathbb{I}\big),\\
\displaystyle\partial_t\big(\rho\chi\big)+\mathrm{div}(\rho \chi \mathbf{u})=-\frac{1}{\epsilon}\frac{\partial f(\rho,\chi)}{\partial\chi}+\frac{\epsilon}{\rho} \Delta\chi,
\end{array}\right.
\end{equation}
where   $\rho=\rho(\mathbf{x},t)$,  $\mathbf{u}=\mathbf{u}(\mathbf{x},t)$ and $\chi=\chi(\mathbf{x},t)$ are the density, the velocity and the concentration difference of the gas-liquid mixture.
The constants $\nu>0, \ \eta\geq0$ are viscosity coefficients, and the constant $\epsilon>0$  is defined as the thickness of the diffuse interface of the gas-liquid mixture.
The potential energy density $f=f(\rho,\chi)$,  satisfying the Ginzburg-Landau double-well potential model (see \cite{HW}, \cite{DLL2013},  \cite{CG2017} and the references therein), follows that:
 \begin{equation}\label{potential energy density}
   f(\rho,\chi)=-3\rho+\frac{8\Theta}{3}\ln\frac{\rho}{3-\rho}+\frac{1}{4}\big(\chi^2-1\big)^2,
 \end{equation}
 with $0<\Theta$ is the  positive constant  related to the ratio of the actual temperature to the critical temperature.
 The pressure $p$ is given by the following van der Waals equation of state (see \cite{V1894}, \cite{HK}, \cite{HW}, \cite{MLW1}, \cite{MLW2}, \cite{HSWZ}, \cite{HLS} and the references therein)
 \begin{equation}\label{the formula of pressure}
p(\rho)=\left\{\begin{array}{llll}
 \displaystyle\rho^2\frac{\partial f}{\partial\rho}=-3\rho^2+\frac{8\Theta\rho}{3-\rho}\quad & \mathrm{if}\ 0\leq\rho<3,\\
 \displaystyle+\infty,\quad & \mathrm{if}\ \rho\geq3.
 \end{array}\right.
\end{equation}
We have the following properties of the pressure $p$:
 \begin{enumerate}
\item[(i)]  $p(\rho)>0$ for $\rho>0$, $p(0)=0$;

\item[(ii)] When $\Theta\geq1$, $p(\rho)$ is a  monotone increasing function. When $0<\Theta<1$,  there exist two positive densities $3>\beta>\alpha>0$ such that $p(\rho)$ is increasing on $[0,\alpha]$ and on $[\beta,3)$, $p(\rho)$ is decreasing on $(\alpha,\beta)$;

 \item[(iii)] $p'(\rho)=\frac{-6(\rho^3-6\rho^2+9\rho-4\Theta)}{(3-\rho)^2}$. When $0<\Theta<1$, there exist a positive density $\gamma$, such that, $p(\gamma)=p(\beta)$,and $p(\rho)>p(\gamma)$ for $\rho>\gamma$, $p$ is increasing on $[0,\gamma]$.
\end{enumerate}
\begin{rmk}The van der Waals state equation  \eqref{the formula of pressure} is  proposed by the Dutch physicist J. D. van der Waals \cite{V1894}. It is a thermodynamic equation of state which is based on the theory that fluids are composed of particles with non-zero volumes, and subject to an  inter-particle attractive force.  Over the critical temperature (i.e. $\Theta\geq1$ in \eqref{the formula of pressure}), this equation of state  is an improvement over the ideal gas law. And what's more, below the critical temperature (i.e. $0<\Theta<1$ in \eqref{the formula of pressure}), this equation is also qualitatively reasonable for the low-pressure gas-liquid states.
\end{rmk}
\begin{rmk}
The concentration difference $\chi$ of the gas-liquid mixture can be understood as
$\chi=\chi_1-\chi_2$, where $\chi_i=\frac{M_i}{M}$  is the mass concentration of the fluid $i~(i=1,2)$,  $M_i$ is the  mass of the components in the representative material volume $V$.
The item of $\epsilon\Big(\nabla\chi\otimes\nabla\chi-\frac{|\nabla\chi|^2}{2}\mathbb{I}\Big)$ in the momentum equation \eqref{3dNSAC} can be seen as an additional stress contribution in the stress tensor. This describes the capillary effect associated with free energy $E_{\mathrm{free}}(\rho,\chi)=\int_\Omega\Big(\frac{\rho}{\epsilon} f(\rho,\chi)+\frac{\epsilon}{2}|\nabla\chi|^2\Big)d\mathbf{x}$, (see \cite{AF2008}, \cite{FPRS2008}, \cite{DLL2013},\cite{CG2017}, \cite{CHMS-2018} and the references therein).
\end{rmk}

There are a lot of works on the well-posedness of the solutions to compressible Navier-Stokes system.  We refer to  the work  of Matsumura-Nishida \cite{MN1980},
Matsumura-Nishihara \cite{MN1985}-\cite{MN1986}, Lions \cite{Lions1998}, Huang-Li-Xin \cite{HLX2012},
  Mei \cite{M1997}-\cite{M1999}, Huang-Li-Matsumura \cite{HLM2010},  Huang-Matsumura-Xin \cite{HMX2006}, Huang-Wang-Wang-Yang \cite{HWWY2015}, Shi-Yong-Zhang \cite{SYZ2016} and the references therein.

The study of interfacial phase changing in mixed fluids can be traced back  to the work by van der Waals (1894).  van der Waals described the interface between two immiscible fluids as a layer in the pioneer  paper \cite{V1894}.  His idea was  successfully applied  by Cahn-Hilliard  \cite{CH1958} and Allen-Cahn \cite{AC1979} to describe the complicated phase separation and coarsening phenomena, the motion of anti-phase boundaries in the mixture respectively.
Lowengrub-Truskinovsky \cite{LT-1998}  added the effect of the motion of the particles  and the interaction with the diffusion into the Cahn-Hilliard equation, and  the Navier-Stokes-Cahn-Hilliard system was put forward.   Blesgen \cite{B1999} then combined the compressible Navier-Stokes system with the modified Allen-Cahn equation to describe the behavior of cavitation in a flowing liquid, which was known as  Navier-Stokes-Allen-Cahn system.  The difference between  Navier-Stokes-Allen-Cahn  system and Navier-Stokes-Cahn-Hilliard system is that,  for the former,  the diffusion fluxes are neglected and  the development of the constitutive equation for mass conversion of any of the considered phases  is focused. This leads to  that the latter conserves the volume fractions while the former does not.

Nowadays, Navier-Stokes-Allen-Cahn system and Navier-Stokes-Cahn-Hilliard system are widely used in the interfacial diffusion problems of fluid mechanics and material science. Comparatively speaking, the numerical treatment to the former is simpler than that of the latter which involves fourth-order differential operators. However, because  the concentration difference $\chi$ in \eqref{3dNSAC} does not preserve overall volume fraction, a Lagrange multiplier  is usually introduced in  \eqref{3dNSAC}$_3$  as a constraint  to conserve the volume, see Yang-Feng-Liu-Shen \cite{YFLS2006}, Zhang-Wang-Mi \cite{ZWM} and the references therein.    Feireisl-Petzeltov$\mathrm{\acute{a}}$-Rocca-Schimperna \cite{FPRS2008} obtained the global existence of weak solutions for the isentropic case, where the  method they used is the framework   introduced by  Lions \cite{Lions1998}. Along the way proposed by Feireisl \textit{et al.}, Ding-Li-Luo \cite{DLL2013} proved the global existence of one-dimensional strong solution  in the  bounded domain for initial density without vacuum states.  Chen-Guo \cite{CG2017}  generalized Ding-Li-Luo's  result  to the case that the initial vacuum is allowed.

However, all the results above are for the ideal fluid. In order to study the gas-liquid phase transition, we need to consider the non-ideal viscous fluid in which,  there is an interval of the density $\rho$ where the pressure $p$ decreases as $\rho$ increases, and the phase transition
takes place. The equations of state \eqref{the formula of pressure} proposed by van der Waals is quite satisfactory in describing this phenomena. Hsieh-Wang \cite{HW} solved the isentropic compressible Navier-Stokes system model by the van der Waals state equation  numerically  by a pseudo-spectral method with a form of artificial viscosity. They  showed that the phase transition depends on the selection of the initial density. He-Liu-Shi \cite{HLS} investigated  the large time behavior for van der Waals fluid in 1-D by using a second order TVD Runge-Kutta splitting scheme combined with Jin-Xin relaxation scheme.
 Mei-Liu-Wong \cite{MLW1,MLW2} studied  Navier-Stokes system  with additional artificial viscosity
 and $p(\rho)=\rho^{-3}-\rho^{-1}$. By using the Liapunov functional method, they proved the existence, uniqueness, regularity and uniform boundedness of the periodic solution in 1-D. Hoff and Khodia \cite{HK} considered the dynamic stability of certain steady-state weak solutions of system (1.1) for compressible van der Waals fluids in 1-D whole space with the small initial disturbance.

In this paper,  we  study the global existence of the solution for the system  \eqref{3dNSAC} with the  van der Waals state equation \eqref{the formula of pressure} in one dimension. More precisely, for general initial conditions without vacuum state, our purpose is to study the existence and uniqueness of global
strong solution for the isentropic  Navier-Stokes-Allen-Cahn systems  \eqref{3dNSAC} even with  large initial data. Moreover we show that the phase variable $\chi$  belongs to the physical interval $[-1,1]$. Some new techniques are developed to  establish the up and low bounds of the density $\rho$, and to treat the non-convexity of the pressure $p(\rho)$, both are crucial steps in the proof.

We now present our main result. The 1-D isentropic Navier-Stokes-Allen-Cahn system in the Euler coordinates is expressed in the following
 \begin{equation}\label{NSAC}
\left\{\begin{array}{llll}
\displaystyle \rho_t+(\rho u)_x=0, \ \ & x\in\mathbb{R},t>0, \\
\displaystyle \rho u_t+\rho uu_x+p_x=\nu u_{xx}-\frac{\epsilon}{2}\big(\chi_x^2\big)_x,\ \ & x\in\mathbb{R},t>0,\\
\displaystyle\rho\chi_t+\rho u\chi_x=-\frac{1}{\epsilon}(\chi^3-\chi)+\frac{\epsilon}{\rho}\chi_{xx},\ \ & x\in\mathbb{R},t>0,
\end{array}\right.
\end{equation}
with the $L$-periodic  boundary value condition:
\begin{equation}\label{periodic boundary for Euler}
\left\{\begin{array}{llll}
(\rho,u,\chi)(x,t)=(\rho,u,\chi)(x+L,t),\ \ & x\in\mathbb{R},t>0,\\
(\rho,u,\chi)\big|_{t=0}=(\rho_0,u_0,\chi_0),\ \ & x\in\mathbb{R}.
\end{array}\right.
\end{equation}
We introduce the Hilbert space $L^2_{\mathrm{per}}$ of  square integrable functions  with the period $L$:
\begin{equation}\label{periodic function sobolev space}
L^2_{\mathrm{per}}=\Big\{g(x)\big|g(x+L)=g(x)\ \mathrm{for\ all}\ x\in\mathbb{R},\ {\mathrm{and}\ } g(x)\in L^2(0,L) \Big\},
\end{equation}
with the norm denoted also by $\|\cdot\|$ (without confusion) which is given by  $\|g\|=(\int_0^L|g(x)|^2dx )^{\frac{1}{2}}$.
$H_{\mathrm{per}}^l \ (l\geq0)$ denotes the  $L_{\mathrm{per}}^2$-functions $g$ on $\mathbb{R}$ whose derivatives $\partial^j_x g,  j=1,\cdots,l$ are $L_{\mathrm{per}}^2$
 functions, with the norm
$  \|g\|_l=(\sum_{j=0}^l\|\partial^j_x g\|^2)^{\frac{1}{2}}$.
The initial and boundary data for the density, velocity and concentration difference of two components are assumed to be:
\begin{equation}\label{initial data of v}
(\rho_0,u_0)\in H_{\mathrm{per}}^1,\ \ \chi_0\in H_{\mathrm{per}}^2;\quad 0<\rho_0<3, \quad-1\leq\chi_0\leq1;
 \end{equation}
\begin{eqnarray}\label{Compatibility condition of chi}
\chi_t(x,0)=-u_0\chi_{0x}+\frac{\epsilon}{\rho_0^2}\chi_{0xx}-\frac{1}{\epsilon\rho_0}\Big(\chi_0^3-\chi_0\Big).
\end{eqnarray}

 \begin{thm} \label{main thm-1}
Assume that $(\rho_0,u_0,\chi_0)$ satisfies    \eqref{initial data of v}-\eqref{Compatibility condition of chi},  then there exists a unique global strong solution $(\rho,u,\chi)$ of the system \eqref{NSAC}-\eqref{periodic boundary for Euler} such that for any $T>0$,
\begin{eqnarray}\label{global solution for periodic boundary}
&&\rho\in L^\infty(0,T;H_{\mathrm{per}}^1)\cap L^2(0,T;H_{\mathrm{per}}^1),\notag\\
&&u\in L^\infty(0,T;H_{\mathrm{per}}^1)\cap L^2(0,T;H_{\mathrm{per}}^2),\notag\\
&&\chi\in L^\infty(0,T;H_{\mathrm{per}}^2)\cap L^2(0,T;H_{\mathrm{per}}^3), \\
&& -1\leq\chi\leq1,\ 0< \rho<3,\ \mathrm{for \ all}\ (x,t)\in\mathbb{R}\times[0,T],\notag
\end{eqnarray}
and
\begin{eqnarray}\label{energy estimate for periodic boundary problem}
\left.\begin{array}{llll}
 \displaystyle\sup_{t\in [0,T]}\big\{\|(\rho,u)(t)\|^2_1+\|\chi\|_2^2\big\}+\int_0^T\big(\|\rho\|_1^2+\|u\|_2^2+\|\chi\|_3^2\big)dt   \leq C,
\end{array}\right.
 \end{eqnarray}
 where $C$ is a positive constant depending only on the initial data and $T$.
\end{thm}

\noindent\begin{rmk}
There are two difficulties to overcome in proving Theorem 1.1. One is the  upper and lower bounds of the density $\rho$, the other is the non-convexity of the pressure. For the former, we use the singularity of pressure  and the energy estimation of $\|\frac{1}{\rho}\|_{L^{\infty}([0,L]\times[0,T])}$.  For the latter, we decompose the pressure according to its convexity. The
results of the Theorem 1.1 are valid  even for large initial data.  They also  match well with the existing numerical studies in \cite{HW} and \cite{HLS}.
\end{rmk}

Moreover, we  consider the following mixed boundary value problem:
\begin{equation}\label{mixed boundary problem}
\left\{\begin{array}{llll}
\displaystyle \rho_t+(\rho u)_x=0,  \\
\displaystyle \rho u_t+\rho uu_x+p_x=\nu u_{xx}-\frac{\epsilon}{2}\big(\chi_x^2\big)_x,\\
\displaystyle\rho\chi_t+\rho u\chi_x=-\frac{1}{\epsilon}(\chi^3-\chi)+\frac{\epsilon}{\rho}\chi_{xx},\\
\displaystyle(u,\chi_x)\big|_{x=0,L}=(0,0),\\
\displaystyle(\rho,u,\chi)\big|_{t=0}=(\rho_0,u_0,\chi_0).
\end{array}\right.
\end{equation}
Similarly,
we have the following existence theorem for the mixed boundary problem \eqref{mixed boundary problem}. The proof will be omitted.
\begin{thm} \label{main thm-2}
Assume that $(\rho_0,u_0,\chi_0)$ satisfies
 \begin{equation}\label{initial data of rho u and chi}
(\rho_0,u_0)\in H^1,\ \ \chi_0\in H^2,\quad 0<\rho_0<3,\quad-1\leq\chi_0\leq1,
 \end{equation}
\begin{eqnarray}\label{Compatibility condition of rho u and chi}
\chi_t(x,0)=-u_0\chi_{0x}+\frac{\epsilon}{\rho_0^2}\chi_{0xx}-\frac{1}{\epsilon\rho_0}\Big(\chi_0^3-\chi_0\Big),
\end{eqnarray}
then there exists a unique global strong solution $(\rho,u,\chi)$ of the system \eqref{mixed boundary problem},such that for any $T>0$,
\begin{eqnarray}\label{global solution for periodic boundary}
&&\rho\in L^\infty(0,T;H^1)\cap L^2(0,T;H^1),\notag\\
&&u\in L^\infty(0,T;H^1)\cap L^2(0,T;H^2),\notag\\
&&\chi\in L^\infty(0,T;H^2)\cap L^2(0,T;H^3), \\
&& -1\leq\chi\leq1,\ 0< \rho<3,\ \mathrm{for \ all}\ (x,t)\in[0,L]\times[0,T],\notag
\end{eqnarray}
and
\begin{eqnarray}\label{energy estimate for mixed boundary problem}
\left.\begin{array}{llll}
 \displaystyle\sup_{t\in [0,T]}\big\{\|(\rho,u)(t)\|^2_1+\|\chi\|_2^2\big\}+\int_0^T\big(\|\rho\|_1^2+\|u\|_2^2+\|\chi\|_3^2\big)dt   \leq C,
\end{array}\right.
 \end{eqnarray}
  where $C$ is a positive constant depending only on the initial data and $T$.
\end{thm}

The outline of this paper is as follows. In Section 2, we first give the local existence of the solution for the system \eqref{NSAC}-\eqref{periodic boundary for Euler}. Then, we give a  series of lemmas which lead us the desired a priori estimates. Finally, Theorem 1.1 is proved by the well-known alternative result and the maximum principle for parabolic equation.

\section {Proofs of the main theorem}
\setcounter{equation}{0}
In this section, we will present the global existence  on strong solution for the  periodic problem \eqref{NSAC}-\eqref{periodic boundary for Euler} .
Firstly, for $\forall m>0$, $M>0$,  $T>0$, we define the periodic solution space:
\begin{eqnarray}\label{periodic function space}
&&X_{\mathrm{per},m,M}([0,T])\equiv\Big\{(\rho,u,\chi)\Big|(\rho,u)\in C^0([0,T];H_{\mathrm{per}}^1),\chi\in C^0([0,T];H_{\mathrm{per}}^2),\qquad\qquad\qquad\notag\\
 &&\qquad\quad\qquad\qquad\qquad\rho\in L^2([0,T];H_{\mathrm{per}}^1), u\in L^2([0,T];H_{\mathrm{per}}^2),\chi\in L^2([0,T];H_{\mathrm{per}}^3),\\
 &&\qquad\quad\qquad\qquad\qquad\qquad\inf_{x\in\mathbb{R},t\in [0,T]}\rho(x,t)\geq m,\sup_{t\in [0,T]}\{\|(\rho,u)\|_1^2,\|\chi\|_2^2\}\leq M
 \Big\}.\notag
\end{eqnarray}

\begin{prop}[Local existence]\label{local existence and uniqueness for approximate periodic solution}
For $\forall m>0$, $M>0$, if $\inf_{x\in\mathbb{R}}\rho_0(x,t)\geq m$, $\|(\rho_0,u_0)\|_1^2$, $\|\chi_0\|_2^2\leq M$,
then there exists a small time $T_*=T_*(\rho_0,u_0,\chi_0)>0$  such that the periodic boundary problem \eqref{NSAC}-\eqref{periodic boundary for Euler} admits a unique solution $(\rho,u,\chi)$ satisfying that $(\rho,u,\chi)\in X_{\mathrm{per},\frac{m}{2},2M}([0,T_*])$.
\end{prop}
\begin{proof} Taking $0<T<+\infty$, for $\forall m>0$, $M>0$, we construct an  iterative sequence $(\rho^{(n)},u^{(n)},\chi^{(n)})$,$n=1,2\cdots\cdots$, satisfying $(\rho^{(0)},u^{(0)},\chi^{(0)})=(v_0,u_0,\chi_0)$,
and the  following iterative scheme
\begin{equation}\label{iterative-NSAH}
\left\{\begin{array}{llll}
\displaystyle \rho^{(n)}_t+(\rho^{(n)} u^{(n-1)})_x=0,  \\
\displaystyle \rho^{(n)} u^{(n)}_t+\rho^{(n)} u^{(n-1)}u^{(n)}_x+(p(\rho^{(n)}))_x=\nu u^{(n)}_{xx}-\frac{\epsilon}{2}\big((\chi^{(n)})_x^2\big)_x,\\
\displaystyle\rho^{(n)}\chi^{(n)}_t+\rho^{(n)} u^{(n-1)}\chi^{(n)}_x=-\frac{1}{\epsilon}((\chi^{(n-1)})^3-\chi^{(n-1)})+\frac{\epsilon}{\rho^{(n)}}\chi^{(n)}_{xx},\\
\displaystyle(\rho^{(n)},u^{(n)},\chi^{(n)})(x,t)=(\rho^{(n)},u^{(n)},\chi^{(n)})(x+L,t),\\
(\rho^{(n)},u^{(n)},\chi^{(n)})(x,0)=\big(\rho_0,u_0,\chi_0\big)(x),
\end{array}\right.
\end{equation}
By using the usual iterative approach (c.f. \cite{CHMS-2018}), we can obtained the local existence of the solution for the periodic boundary problem \eqref{NSAC}-\eqref{periodic boundary for Euler},  the details are omitted.
\end{proof}
Now we will prove  the global existence and uniqueness of the solution
for the periodic boundary problem \eqref{NSAC}-\eqref{periodic boundary for Euler}.
 Setting
\begin{equation}\label{mu}
  \mu=\frac{1}{\epsilon}(\chi^3-\chi)-\frac{\epsilon}{\rho}\chi_{xx}.
\end{equation}
From the physical point of view, the functional $\mu$ in \eqref{mu} can be understood as the chemical potential. The basic energy equality is presented below. From  the definition of the pressure $p$ in \eqref{the formula of pressure}, we fix a positive reference density $\tilde \rho$ satisfying (see the properties of $p$)
\begin{equation}\label{reference density}
0<\tilde\rho<\gamma<3,
\end{equation}
and define
\begin{equation}\label{Phi}
\Phi(\rho)=\rho\int_{\tilde{\rho}}^{\rho}\frac{p(s)-p(\tilde{\rho})}{s^2}ds.
\end{equation}
Noting that
\begin{equation*}
  \Phi'(\rho)=\frac{\Phi(\rho)+p(\rho)-p(\tilde\rho)}{\rho},\qquad \mathrm{and}\qquad \Phi''(\rho)=\frac{p'(\rho)}{\rho},
\end{equation*}
then $\Phi(\tilde \rho)=\Phi'(\tilde\rho)=0$, and so that, there exist positive constants $c_1,c_2>0$ such that
\begin{equation}\label{positive definite}
 c_1(\rho-\tilde\rho)^2\leq\Phi(\rho)\leq c_2(\rho-\tilde\rho)^2.
\end{equation}
Moreover, combining with the mass conservation equation \eqref{NSAC}$_1$, one gets
\begin{equation}\label{renormalization mass conservation equation}
\Phi(\rho)_t+\big(\Phi(\rho)u\big)_x+\big(p(\rho)-p(\tilde{\rho})\big)u_x=0.
\end{equation}
Taking advantage of the local existence result Proposition \ref{local existence and uniqueness for approximate periodic solution}, we know that there exists a unique strong solution of the system  \eqref{NSAC}-\eqref{periodic boundary for Euler} for $T$ small enough.  By using the well-known alternative result, and the maximum principle for parabolic equation (see \cite{P2005}), it suffices to show the following  a priori estimate.
\begin{prop}[A priori estimate]\label{a priori estimate proposition for periodic boundary problem}
 Assume that $(v_0,u_0,\chi_0)$ satisfies  \eqref{initial data of v}-\eqref{Compatibility condition of chi}, let $(\rho,u,\chi)\in X_{\mathrm{per},m,M}([0,T])$ be a local solution for a given $T>0$, then  there exists a  positive constant $C$,  such that
 \begin{eqnarray}\label{a priori estimate}
\left.\begin{array}{llll}
 \displaystyle\sup_{t\in [0,T]}\big\{\|(\rho,u)(t)\|^2_1+\|\chi\|_2^2\big\}+\int_0^T\big(\|\rho\|_1^2+\|u\|_2^2+\|\chi\|_3^2\big)dt   \leq C.
\end{array}\right.
 \end{eqnarray}
\end{prop}
Proposition 2.2 can be obtained by the following series of  lemmas.

\begin{lem}\label{lem of lower estimate}
Under the assumption of Proposition \ref{a priori estimate proposition for periodic boundary problem},  for $\forall T>0$,
  it holds that
\begin{eqnarray}\label{the first energy inequality}
&&\int_0^L\Big(\rho u^2+\Phi(\rho)+\chi_x^2+\rho(\chi^2-1)^2\Big)dx+\int_0^T\int_0^L\Big(\mu^2+u_x^2\Big)dxdt\leq C,
\end{eqnarray}
where $\mu$ is defined in \eqref{mu}.
\end{lem}
\begin{proof}
Multiplying Eq.\eqref{NSAC}$_2$ by $u$ and Eq.\eqref{NSAC}$_3$ by $\mu$, integrating the resultant equations over $[0,L]$ and adding them up, one has
\begin{equation}\label{basic energy equality}
\frac{d}{dt}\int_0^L\big(\frac{\rho u^2}{2}+\frac{\epsilon\chi_x^2}{2}+\frac{\rho(\chi^2-1)^2}{4\epsilon}\big)dx+\int_0^L\Big(\mu^2+\nu u_x^2+up_x(\rho)\Big)dxdt=0.
\end{equation}
Integrating \eqref{renormalization mass conservation equation} and adding the result to \eqref{basic energy equality}, one then gets
\begin{equation}\label{the basic energy inequality}
\frac{d}{dt}\int_0^L\Big(\frac{\rho u^2}{2}+\frac{\epsilon}{2}\chi_x^2+\Phi(\rho)+\frac{\rho(\chi^2-1)^2}{4\epsilon}\Big)dx+\int_0^L\Big(\mu^2+\nu u^2_x\Big)dxd\tau=0.
\end{equation}
Integrating \eqref{the basic energy inequality} over $[0,T] $, one has
\begin{equation}\label{the basic energy inequality for density velocity and concentration difference}
\sup_{t\in[0,T]}\int_0^L\Big(\frac{\rho u^2}{2}+\frac{\epsilon}{2}\chi_x^2+\Phi(\rho)+\frac{\rho(\chi^2-1)^2}{4\epsilon}\Big)dx+\int_0^T\int_0^L\Big(\mu^2+\nu u^2_x\Big)dxd\tau=E_0,
\end{equation}
where $E_0=\int_0^L\big(\frac{1}{2}\rho_0 u_0^2+\frac{\epsilon}{2}\chi_{0x}^2+\Phi(\rho_0)+\frac{\rho_0}{4\epsilon}(\chi_0^2-1)^2\big)dx$.
The proof  is obtained.
\end{proof}

\begin{lem}\label{lem of lower estimate for chi}
Under the assumption of Proposition \ref{a priori estimate proposition for periodic boundary problem},  for $\forall T>0$,
  it holds that
\begin{eqnarray}\label{sup of concentration}
\|\chi\|_{L_{\mathrm{per}}^{\infty}}\leq C.
\end{eqnarray}
\end{lem}

\begin{proof}
Integrating the mass equation \eqref{NSAC}$_1$ over $[0,L]\times[0,t]$, one has
\begin{equation}\label{mass conservation}
 \int_0^L\rho(x,t)dx=\int_0^L\rho_0(x)dx.
\end{equation}
By Lemma 2.1, we then have
\begin{equation}\label{inequality of  chi}
  \int_0^L\rho\chi^4dx\leq2\int_0^L\rho\chi^2dx-\int_0^L\rho dx+C_1\leq\frac{1}{2}\int_0^L\rho\chi^4dx+C.
\end{equation}
Therefore
\begin{equation}\label{L4 L1 of chi}
  \int_0^L\rho\chi^4dx\leq C,\ \ \ \int_0^L\rho\chi dx \leq \int_0^L\rho\chi^4dx+\int_0^L\rho dx\leq C.
\end{equation}
From \eqref{the first energy inequality}, one has
\begin{eqnarray}\label{sup of chi}
|\chi(x,t)|&=&\frac{1}{\int_0^L\rho_0dx}\Big|\chi(x,t)\int_0^L\rho(y,t)dy\Big|\notag\\
&\leq&\frac{1}{\int_0^L\rho_0dx}\Big(\big|\int_0^L\big(\chi(x,t)-\chi(y,t)\big)\rho(y,t)dy\big|+\big|\int_0^L\chi(y,t)\rho(y,t)dy\big|\Big)\notag\\
&\leq&\frac{1}{\int_0^L\rho_0dx}\Big(\big|\int_0^L\rho(y,t)\big(\int_y^x\chi_s(s,t)ds\big)dy\big|+\big|\int_0^L\chi(y,t)\rho(y,t)dy\big|\Big)\notag\\
&\leq&\frac{1}{\int_0^L\rho_0dx}\int_0^L|\chi_x|dx\int_0^L\rho(y,t)dy+C_1\leq
C.
\end{eqnarray}
The proof  is completed.
\end{proof}

\begin{lem}\label{lem of sup estimate for rho}
Under the assumption of Proposition \ref{a priori estimate proposition for periodic boundary problem},  for $\forall T>0$,
  it holds that
\begin{eqnarray}\label{sup of density}
\|\rho\|_{L_{\mathrm{per}}^{\infty}([0,L]\times[0,T])}<3,\ \ \ \int_0^T\int_0^L\chi_{xx}^2dx\leq C.
\end{eqnarray}
\end{lem}
\begin{proof}

Observing   Lemma 2.1, one has
\begin{eqnarray}\label{first energy inequality}
\sup_{t\in[0,T]}\int_0^L\Phi(\rho)dx\leq E_0=\int_0^L\big(\frac{1}{2}\rho_0 u_0^2+\frac{\epsilon}{2}\chi_{0x}^2+\Phi(\rho_0)+\frac{\rho_0}{4\epsilon}(\chi_0^2-1)^2\big)dx.
\end{eqnarray}
From the definitions of \eqref{Phi} and  \eqref{the formula of pressure},  one gets
\begin{equation}\label{delta limit}
\lim_{\delta\rightarrow0}\mathrm{mes}\big\{(x,t)\in[0,L]\times[0,T]\big|\rho(x,t)\geq3-\delta\big\}=0,
\end{equation}
thus
\begin{equation}\label{upper bound of density}
 \|\rho(x,t)\|_{L^{\infty}([0,L]\times[0,T])}<3.
\end{equation}
Moreover, from the equation \eqref{mu} and the energy inequalities \eqref{L4 L1 of chi}, \eqref{sup of chi},  one obtains
$$\int_0^T\int_0^L\chi_{xx}^2dx=\int_0^T\int_0^L\Big(\rho(\chi^3-\chi)-\rho\mu\Big)^2dx\leq C.$$
The proof  is completed.
\end{proof}

\begin{lem}\label{lem of inf estimate for rho}
Under the assumption of Proposition \ref{a priori estimate proposition for periodic boundary problem},  for $\forall T>0$,
  it holds that
\begin{eqnarray}\label{inf of density}
\sup_{t\in[0,T]}\|\rho_x\|_{L^2_{\mathrm{per}}}\leq C, \ \ \ \|\frac{1}{\rho}\|_{L_{\mathrm{per}}^{\infty}([0,L]\times[0,T])}\leq C.
\end{eqnarray}
\end{lem}
\begin{proof}
From the mass conservation equation \eqref{NSAC}$_1$, one has
\begin{eqnarray}\label{The relation between density and velocity}
 u_{xx}&=&-\big[\frac{1}{\rho}\big(\rho_t+\rho_x u\big)\big]_x=
 \big[(-\ln\rho)_t+\rho u(\frac{1}{\rho})_x\big]_x=[-(\ln\rho)_x]_t+[\rho u(\frac{1}{\rho})_x\big]_x\notag\\
 &=&\big[\rho(\frac{1}{\rho})_x]_t+[\rho u(\frac{1}{\rho})_x\big]_x=
 \rho(\frac{1}{\rho})_{xt}+\rho u(\frac{1}{\rho})_{xx}+[\rho_x(\frac1\rho)_t+(\rho u)_x(\frac1\rho)_x]\notag\\
 &=&\rho(\frac{1}{\rho})_{xt}+\rho u(\frac{1}{\rho})_{xx}-\frac{\rho_x}{\rho^2}\big[\rho_t+(\rho u)_x\big]=\rho(\frac{1}{\rho})_{xt}+\rho u(\frac{1}{\rho})_{xx}.
\end{eqnarray}
 Substituting \eqref{The relation between density and velocity} into the momentum equation \eqref{NSAC}$_2$, one gets
\begin{equation}\label{the other form for NSAC-2}
  (\rho u)_t+(\rho u^2)_x+p'(\rho)\rho_x=\nu\big[\rho\frac{d}{dt}(\frac{1}{\rho})_x+\rho u(\frac{1}{\rho})_{xx}\big]-\frac{\epsilon}{2}\big(\chi_x^2\big)_x,
\end{equation}
Multiplying  \eqref{the other form for NSAC-2} by $(\frac{1}{\rho})_x$, and integrating over $[0,L]$, further
\begin{eqnarray}\label{the basic energy equality-2 for density}
&&\frac{d}{dt}\int_0^L\big(\frac\nu2\rho\big|\big(\frac{1}{\rho}\big)_x\big|^2-\rho u(\frac{1}{\rho})_x \big)dx+\int_0^L \frac{p'(\rho)}{\rho^2}\rho_x^2dx\notag\\
&&=-\int_0^L\rho u(\frac1\rho)_{xt}dx+\int_0^L(\rho u^2)_x(\frac1\rho)_xdx+\frac{\epsilon}{2}\int_0^L\big(\chi_x^2\big)_x(\frac{1}{\rho})_xdx\notag\\
&&=\int_0^L\Big((\rho u)_x(-\frac{\rho_t}{\rho^2})+(\rho u^2)_x(-\frac{\rho_x}{\rho^2})\Big)dx+\epsilon\int_0^L\chi_x\chi_{xx}(\frac1\rho)_xdx\\
&&=\int_0^L u_x^2dx+\epsilon\int_0^L\chi_x\chi_{xx}(\frac1\rho)_xdx\notag\\
&&\leq\int_0^L u_x^2dx+\epsilon\Big(\|\frac1\rho\|_{L_{\mathrm{per}}^{\infty}}+\int_0^L\rho\big|(\frac1\rho)_x\big|^2dx\Big)\|\chi_{xx}\|_{L_{\mathrm{per}}^2}^2.\notag
\end{eqnarray}
In view of the mean value theorem, there exists $a(t)\in [0,L]$ satisfying $\rho(a(t),t)=\frac{1}{L}\int_0^L\rho_0dx$, so that
\begin{eqnarray}
\frac{1}{\rho(x,t)}&=&\frac{1}{\rho(x,t)}-\frac{1}{\rho(a(t),t)}+\frac{1}{\rho(a(t),t)}\notag\\
&=&\int_{a(t)}^x\big(\frac{1}{\rho(y,t)}\big)_ydy+\frac{L}{\int_0^L\rho_0dx}\notag\\
&\leq&\int_0^L\big|\frac{\rho_x(x,t)}{\rho^2(x,t)}\big|dx+\frac{L}{\int_0^L\rho_0dx}\\
&\leq&\big(\int_0^L\frac1\rho dx\big)^{\frac12}\Big(\int_0^L\frac{\rho^2_x(x,t)}{\rho^3(x,t)}dx\Big)^{\frac12}+\frac{L}{\int_0^L\rho_0dx}\notag\\
&\leq&\frac{1}{2}\big\|\frac{1}{\rho}\big\|_{L_{\mathrm{per}}^{\infty}}+
\frac{L}{2}\int_{0}^L\rho\big|\big(\frac{1}{\rho}\big)_x\big|^2dx+\frac{L}{\int_0^L\rho_0dx},\notag
\end{eqnarray}
then one has the Sobolev inequality about $\frac{1}{\rho}$,
\begin{equation}\label{the sobolev inequality of 1/rho}
 \big\|\frac{1}{\rho}\big\|_{L_{\mathrm{per}}^{\infty}}\leq  L\int_{0}^L\rho\big|\big(\frac{1}{\rho}\big)_x\big|^2dx+\frac{2 L}{\int_0^L\rho_0dx}.
 \end{equation}
Substituting the above expression into the inequality \eqref{the basic energy equality-2 for density},  one gets
\begin{eqnarray}\label{the energy equality for derivative of density}
 &&\frac{d}{dt}\int_0^L\big(\frac\nu2\rho\big|\big(\frac{1}{\rho}\big)_x\big|^2-\rho u(\frac{1}{\rho})_x \big)dx+\int_0^L \frac{p'(\rho)}{\rho^2}\rho_x^2dx\notag\\
&&\leq\int_0^L u_x^2dx+\epsilon\Big((L+1)\int_{0}^L\rho\big|\big(\frac{1}{\rho}\big)_x\big|^2dx+\frac{2 L}{\int_0^L\rho_0dx}\Big)\|\chi_{xx}\|_{L_{\mathrm{per}}^2}^2.
\end{eqnarray}
Setting
\begin{eqnarray}\label{the set of derivative for p}
 &&A_{\mathrm{increase}}(t)=\big\{x\in[0,L]\big|0\leq\rho(x,t)<\alpha\big\}\cup\big\{x\in[0,L]\big|\beta<\rho\leq M\big\},\\
 &&A_{\mathrm{decrease}}(t)=\big\{x\in[0,L]\big|\alpha\leq\rho(x,t)\leq\beta\big\},
\end{eqnarray}
then multiplying \eqref{the energy equality for derivative of density} by $\frac{\nu}{2}$, and adding up  \eqref{the basic energy inequality},   one gets
\begin{eqnarray}\label{key inequality-2 for density}
 &&\frac{d}{dt}\int_0^L\big(\frac{\mu^2\rho}{4}\big|\big(\frac{1}{\rho}\big)_x\big|^2-\frac{\mu\rho u}{2}(\frac{1}{\rho})_x +\frac{\rho u^2}{2}+\Phi(\rho)+\frac{\rho(\chi^2-1)^2}{4\epsilon}+\frac{\epsilon\chi_x^2}{2}\big)dx\notag\notag\\
&& \ +\int_{A_{\mathrm{increase}}(t)}\rho p'(\rho)(\frac1\rho)_x^2dx+\int_0^L\big(\mu^2+\frac{\nu}{2}u_x^2\big)dx\\
&&\leq\frac{\epsilon\nu}{2}\Big((L+1)\int_{0}^L\rho\big|\big(\frac{1}{\rho}\big)_x\big|^2dx+\frac{2 L}{\int_0^L\rho_0dx}\Big)\|\chi_{xx}\|_{L_{\mathrm{per}}^2}^2-\int_{A_{\mathrm{decrease}}(t)}\rho p'(\rho)(\frac1\rho)_x^2dx\notag\\
&&\leq\frac{\epsilon\nu}{2}\Big((L+1)\int_{0}^L\rho\big|\big(\frac{1}{\rho}\big)_x\big|^2dx+\frac{2 L}{\int_0^L\rho_0dx}\Big)\|\chi_{xx}\|_{L_{\mathrm{per}}^2}^2+\frac{6(27-4\Theta)}{(3-\beta)^2}\int_0^L\rho\big|(\frac1\rho)_x\big|^2dx.\notag
\end{eqnarray}
Integrating the inequality \eqref{key inequality-2 for density}  over $[0,T]$, applying Lemma 2.2-2.3 and combining with Gronwall's inequality, one obtains
\begin{eqnarray}
 &&\int_0^L\big(\rho\big|\big(\frac{1}{\rho}\big)_x\big|^2+\rho u^2+(\rho-\tilde\rho)^2+\rho(\chi^2-1)^2+\chi_x^2\big)dx\notag+\int_0^T\int_0^L\big(\mu^2+u_x^2\big)dx\leq C.
\end{eqnarray}
In view of \eqref{the sobolev inequality of 1/rho}, combining with $\int_0^L\rho\big|\big(\frac{1}{\rho}\big)_x\big|^2dx\geq\frac{1}{\|\rho\|_{L_{\mathrm{per}}^\infty}^3}\int_0^L\rho_x^2dx$, the proof of Lemma 2.4 is completed.
\end{proof}

The estimate of the higher order derivatives for the phase parameter $\chi$ and the velocity $u$ can be obtained in a simpler way then in Lemma 2.1-Lemma 2.4.
\begin{lem}\label{higher order derivatives for chi and u}
Under the assumption of Proposition \ref{a priori estimate proposition for periodic boundary problem},  for $\forall T>0$,
  it holds that
\begin{equation}\label{twice order derivatives for chi}
\sup_{t\in[0,T]}\big(\|\chi_{t}\|^2_{L^2_{\mathrm{per}}}+\|\chi_{xx}\|^2_{L^2_{\mathrm{per}}}\big)+
\int_0^T\int_0^L\big(\chi_{xt}^2+\chi^2_{t}+\chi_{xxx}^2\big)dxdt\leq C,
\end{equation}
\begin{equation}\label{twice order derivatives for u}
\sup_{t\in[0,T]}\|u_{x}\|^2_{L^2_{\mathrm{per}}}+\int_0^T\int_0^L\big(u_t^2+u_{xx}^2\big)dxdt\leq C.
\end{equation}
\end{lem}
\begin{proof}
For the sake of convenience, we introduce the Lagrange coordinate system below:
 \begin{equation}\label{lagrange coordinate}
y=\int_0^x\rho(s,t) ds,\ \ t=t;\qquad v=\frac{1}{\rho}.
\end{equation}
Integrating \eqref{NSAC} over $[0,R]\times[0,t]$ and using the boundary condition   \eqref{periodic boundary for Euler}, we have
\begin{equation}\label{conservation}
   \frac{1}{L}\int_0^L\rho dx=\frac1L\int_0^L\rho_0dx:=\bar\rho. \end{equation}
Setting
\begin{equation}\label{length of period}
\tilde{L}:=\bar\rho L,
\end{equation}
then the system \eqref{NSAC} can be reduced into
 \begin{equation}\label{Lagrange-NSAC-modified}
\left\{\begin{array}{llll}
\displaystyle v_t-u_y=0,  \ \ & y\in\mathbb{R},t>0,\\
\displaystyle u_t+p_y=\nu \big(\frac{u_{y}}{v}\big)_y-\epsilon\big(\frac{\chi_y^2}{v^2}\big)_y,\ \ & y\in\mathbb{R},t>0,\\
\displaystyle \chi_t=-\frac{v}{\epsilon}(\chi^3-\chi)+\epsilon v\big(\frac{\chi_y}{v}\big)_y,\ \ & y\in\mathbb{R},t>0,\\
(v,u,\chi)(y,t)=(v,u,\chi)(y+\tilde{L},t),\ \ & y\in\mathbb{R},t>0,\\
(v,u,\chi)\big|_{t=0}=(v_0,u_0,\chi_0),\ \ & y\in\mathbb{R}.
\end{array}\right.
\end{equation}
From \eqref{Lagrange-NSAC-modified}$_3$, one has
\begin{equation}\label{the time dereviative of chi}
 \chi_t=-\frac{1}{\epsilon}v(\chi^3-\chi)+\frac{\chi_{yy}}{v}-\frac{2\chi_yv_y}{v^2},
\end{equation}
then \eqref{the time dereviative of chi} and Lemma 2.1-2.4 implies that
\begin{equation}\label{the first energy inequality in Lagrange coordinate}
\displaystyle\int_0^{\tilde{L}}\Big(u^2+v^2_y+v^2+\chi_y^2+(\chi^2-1)^2\Big)dy+\int_0^T\int_0^{\tilde{L}}\Big(\mu^2+u_y^2+\chi_t^2+\chi^2_{yy}\Big)dy\leq C,\end{equation}
\begin{equation}\label{upper and lower bound in Lagrange coordinate}
\displaystyle 0<c\leq v\leq C<+\infty,\ \ and \ \ 0\leq\chi\leq C,
\end{equation}
and
\begin{equation}\label{the relationship between chi-yy with chi-t}
  \int_0^{\tilde{L}}\chi_{yy}^2dy\leq C\Big(\int_0^{\tilde{L}}\chi_t^2dy+1\Big).\end{equation}
Differentiating \eqref{the time dereviative of chi} with respect to $t$, one gets
\begin{equation}\label{time derivative of equation 3}
  \chi_{tt}=-\frac{v_t}{\epsilon}(\chi^3-\chi)-\frac{v}{\epsilon}(3\chi^2-1)\chi_t+\epsilon v_t\Big(\frac{\chi_y}{v^2}\Big)_y+\epsilon v\Big(\frac{\chi_y}{v^2}\Big)_{yt}.
\end{equation}
Multiplying \eqref{time derivative of equation 3} by $\chi_t$, and integrating it over $[0,\tilde{L}]$ with respect of $y$, one obtains
\begin{eqnarray}\label{bound of chi-t}
&&\frac{1}{2}\frac{d}{dt}\int_0^{\tilde{L}}\chi_t^2dy+ \epsilon\int_0^{\tilde{L}}\frac{\chi_{yt}^2}{v}dy\notag\\
&&=-\frac{1}{\epsilon}\int_0^{\tilde{L}}\big((\chi^3-\chi)u_y\chi_t+v(3\chi^2-1)\chi_t^2\big)dy+\epsilon\int_0^{\tilde{L}}u_y\Big(\frac{\chi_y}{v^2}\Big)_y\chi_tdy\notag\\
&&\ \ \ \ \ -\epsilon\int_0^{\tilde{L}}v_y\Big(\frac{\chi_y}{v^2}\Big)_{t}\chi_tdy+\epsilon\int_0^{\tilde{L}}\frac{2}{v^2}\chi_yu_y\chi_{yt}dy\notag\\
&&=I_1+I_2+I_3+I_4.
\end{eqnarray}
Following from Sobolev inequality and Lemma 2.1-2.4 and \eqref{the first energy inequality in Lagrange coordinate}-\eqref{upper and lower bound in Lagrange coordinate}, one deduces
\begin{equation}\label{I1}
\big|I_1\big|\leq C_1\big(\|u_y\|^2+\|\chi_t\|^2\big),
\end{equation}
\begin{eqnarray}\label{I2}
\big|I_2\big|&\leq& C\Big(\int_0^{\tilde{L}}\big|u_y\chi_{yy}\chi_t\big|dy+\int_0^{\tilde{L}}\big|u_y\chi_yv_y\chi_t\big|dy\Big)\notag\\
&\leq& C\Big(\|\chi_t\|_{L^{\infty}}\|u_y\|\|\chi_{yy}\|+\|\chi_t\|_{L^{\infty}}\|\chi_y\|_{L^{\infty}}\|u_y\|\|v_y\|\Big)\notag\\
&\leq& C\Big(\|\chi_t\|^2\|u_y\|+\|\chi_t\|^2\|u_y\|^{\frac43}+\|\chi_t\|^{\frac32}\|u_y\|+\|\chi_t\|^{\frac43}\|u_y\|^{\frac43}+\|\chi_t\|\|u_y\|+\|\chi_t\|^{\frac23}\|u_y\|^{\frac43}\Big)\notag\\
& &+\frac{\epsilon}{4}\|\chi_{yt}\|^2,
\end{eqnarray}
and
\begin{eqnarray}\label{I3}
\big|I_3\big|+\big|I_4\big|&\leq& C\int_0^{\tilde{L}}\big(|v_y\chi_{yt}\chi_t|+|v_y\chi_yu_y\chi_t|+|\chi_yu_y\chi_{yt}|\Big)dy\notag\\
&\leq&C\Big(\|\chi_t\|\|u_y\|^2+\|\chi_t\|^2\Big)+\frac{\epsilon}{4}\|\chi_{yt}\|^2.
\end{eqnarray}
Substituting \eqref{I1}--\eqref{I3} into \eqref{bound of chi-t}, applying the Gronwall's inequality, one drives
\begin{equation}\label{bound of chi-t and chi-yt}
\int_0^{\tilde{L}}\chi_t^2dy+\int_0^T\int_0^{\tilde{L}}\chi_{yt}^2dy\leq C.
\end{equation}
Combining with \eqref{the relationship between chi-yy with chi-t}, one gets
\begin{equation}\label{the bound of the second derivative of concentration}
\int_0^{\tilde{L}}\chi_{yy}^2dy\leq C.
\end{equation}
It holds that
 \begin{equation}\label{twice order derivatives for chi in lagrange coordinates}
\int_0^{\tilde{L}}(\chi_t^2+\chi_{yy}^2)dy+\int_0^T\int_0^{\tilde{L}}\big(\chi_{yt}^2+\chi^2_{t}+\chi_{yy}^2\big)dy\leq C.
\end{equation}
Multiplying \eqref{Lagrange-NSAC-modified}$_2$  by $-u_{yy}$, integrating over $[0,{\tilde{L}}]$ by parts, by using Sobolev inequality, Lemma 2.1-2.4  and \eqref{twice order derivatives for chi}, one obtains
\begin{eqnarray}\label{the second derivative of velocity}
 &&\big(\frac{1}{2}\int_0^{\tilde{L}}u^2_ydy\big)_t+\nu\int_0^{\tilde{L}}\frac{u_{yy}^2}{v}dy\notag\\
 &&=\int_0^{\tilde{L}} u_{yy}(p_\delta)'_vv_ydy+\int_0^{\tilde{L}}\frac{ u_{yy}u_yv_y}{v^{2}}dy+\int_0^{\tilde{L}}\frac{2\epsilon\chi_y\chi_{yy}u_{yy}}{v^3}dy-\int_0^{\tilde{L}}\frac{3\epsilon\chi_y^2v_y u_{yy}}{v^4}dy\notag\\
 &&\leq C\big(\|u_y\|^2+1\big)+\frac{\nu}{2}\int_0^{\tilde{L}}\frac{u_{yy}^2}{v}dy.
\end{eqnarray}
Thus  it holds that
 \begin{equation}\label{twice order derivatives for u in lagrange coordinates}
 \int_0^{\tilde{L}}u^2_ydy+\int_0^T\int_0^{\tilde{L}}u_{yy}^2dy\leq C.
\end{equation}
Let's go back to the Euler coordinates, by using \eqref{twice order derivatives for chi in lagrange coordinates}, \eqref{twice order derivatives for u in lagrange coordinates}, combining with $\chi_{xxx}=2\rho\rho_x\chi_t+\rho^2\chi_{xt}+2\rho\rho_xu\chi_x+\rho^2u_x\chi_x+\rho^2u\chi_{xx}+\rho_x(\chi^3-\chi)+\rho(3\chi^2-1)\chi_x$, one has
\begin{equation}
\sup_{t\in[0,T]}\big(\|\chi_{t}\|^2_{L^2_{\mathrm{per}}}+\|\chi_{xx}\|^2_{L^2_{\mathrm{per}}}\big)+
\int_0^T\int_0^L\big(\chi_{xt}^2+\chi^2_{t}\big)dxdt\leq C,
\end{equation}
\begin{equation}
\sup_{t\in[0,T]}\|u_{x}\|^2_{L^2_{\mathrm{per}}}+\int_0^T\int_0^Lu_{xx}^2dxdt\leq C,
\end{equation}
and
\begin{equation}\label{third derivative for chi}
 \int_0^T\int_0^L\chi_{xxx}^2dxdt\leq C.
\end{equation}
Furthermore, by using $u_t=-(p_\delta)_y+\nu \big(\frac{u_{y}}{v}\big)_y-\epsilon\big(\frac{\chi_y^2}{v^2}\big)_y$, one obtains
\begin{equation}\label{t derivative for u}
 \int_0^T\int_0^Lu_t^2dxdt\leq C.
\end{equation}
Then proof of Lemma 3.5 is achieved.
\end{proof}

From Lemma 2.1-Lemma 2.5, Proposition 2.2 is obtained, and the proof of Theorem 1.1 is completed.

\bigskip

\noindent {\textbf{Acknowledgments:}   The research of M.Mei was supported in part by NSERC 354724-2016 and FRQNT grant 256440. The research of  X. Shi was  partially supported by National Natural Sciences Foundation of China No. 11671027 and 11471321. The  work of X.P. Wang  was supported in part by the Hong Kong Research Grants Council (GRF grants 16324416, 16303318 and NSFC-RGC joint research grant N-HKUST620/15).}

\end{document}